\documentclass[11pt,a4]{article}

\usepackage{fullpage}
\usepackage[cp1251]{inputenc}
\usepackage[T2A]{fontenc}       
\usepackage[russian,english]{babel}
\usepackage{caption}
\usepackage{graphicx}
\usepackage{amsmath}
\usepackage{amssymb}
\usepackage{amsthm}
\usepackage{color}
\usepackage{ifthen}
\usepackage{indentfirst}

\DeclareMathOperator{\arccot}{arccot}

\newlength{\ps}
\newcommand{\Skip}[1]{}

\newcommand{\up}{\vspace{-.3\baselineskip}}

\newtheorem{thm}{Theorem}
\newtheorem*{cor}{Corollary}

\theoremstyle{definition}

\newsavebox{\TmpBox}
\newcommand{\tmp}{}
\newlength{\tmplength}                    %

\newcommand{\Equa}[2]{%
\begin{equation}#2\label{#1}\end{equation}}
\newcommand{\equa}[1]{\[ #1 \]}

\newcommand{\Diam}[1]{\varnothing_{#1}}
\newcommand{\abs}[1]{\left\lvert#1\right\rvert}

\newcommand{\sgn}{\mathop\mathrm{sgn}\nolimits}
\newcommand{\Deg}[1]{{\ifmmode{#1}^\circ\else${#1}^\circ$\fi}}
\newcommand{\e}{{\mathrm{e}}}
\newcommand{\Exp}[1]{\e^{#1}}
\newcommand{\nvec}[1]{\mathbf{n}\left(#1\right)}

\newcommand{\iu}{{\mathrm{i}}}
\newcommand{\Int}[4]{\displaystyle\int\limits_{#1}^{#2}{#3}\,{\mathrm d}#4}

\newcommand{\Arc}[1]{\displaystyle{\buildrel\,\,\frown\over{#1}}}
\newcommand{\Brack}[1]{\left[#1\right]}
\newcommand{\Brace}[1]{\left\{#1\right\}}
\newcommand{\Skobki}[1]{\left(#1\right)}
\renewcommand{\le}{\leqslant}
\renewcommand{\ge}{\geqslant}

\newcommand{\Eqref}[1]{\buildrel\eqref{#1}\over=}
\newcommand{\So}{\quad\Longrightarrow\quad}

\newcommand{\myfrac}[2]{%
{\ifmmode{}^{#1}\!/_{\!#2}\else${}^{#1}\!/_{\!#2}$\fi}}

\newcommand{\Kl}[1]{{\ifmmode{\mathcal{K}_{#1}}\else$\mathcal{K}_{#1}$\fi}}

\newcommand{\tauJ}{\tau_{{}_J}}
\newcommand{\Barc}{\mathcal{B}}

\def\FigDir{}
\newcommand{\Figref}[1]{\ref{fig:#1}}
 
\newcommand{\RefFig}[2][]{Fig.\,\Figref{#2}{#1}} 
\newcommand{\Infigw}[2]{
\includegraphics[width=#1]{\FigDir#2.eps}}

\newcommand{\Infig}[3]{%
\Infigw{#1}{#2}%
\caption{{\small #3}}\label{fig:#2}%
}

\newcommand{\Pfig}[3]{
\centering\Infig{#1}{#2}{#3}}

\newcommand{\quo}[1]{\glqq#1\grqq}

\begin{document}

\author{Alexey Kurnosenko}
\title{On approximation of planar curves by circular arcs with~length~preservation}

\maketitle{}
\begin{abstract}
The method for approximation of a planar curve by circular arcs with length preservation,
proposed by I.\,Kh.\,Sabitov and A.\,V.\,Slovesnov,
is analyzed.
We extend the applicability of the method,
and consider some corollaries,
not related to the approximation problem.
Inequalities for the length of a convex spiral arc
with prescribed two-point G$^1$ or G$^2$ Hermite data are derived.
We propose a scheme of computer modelling to explore properties of planar curves.
As an example, closeness of ovals is tested,
leading to some conjectures about closeness conditions.
\end{abstract}

This note further develops the subject of article~\cite{Sabitov},
whose motivation is sufficiently reasoned
by the authors.
In~\cite{Sabitov} a segment of monotone curvature within some given curve
is approximated by a biarc curve, sharing end points and end tangents
with the segment.
In Computer-Aided Design (CAD) applications this is a well-known problem of approximation with
given two-point G$^1$ Hermite data.
Additional condition is required to select a solution among
a family of biarcs.
E.\,g., the condition of minimal curvature jump in the join point
could be imposed.
To the author's knowledge, the condition of lengths preservation,
proposed in~\cite{Sabitov}, is new for this problem.

Here we extend both the variety of curves, to which the method could be applied
(see comments to Theorem\,1),
and fields of its application, not related to the approximation problem
(Sections 3,\,4).
First, we precise some terminology.\up
%
\begin{itemize}
\item
As the customary to CAD,
we designate curves with monotone curvature as {\em spirals}~\cite{PomiMain}.\up
\item
{\em Biarc curve}~\cite{PomiMain,Bolton}
is the curve, composed of two non-closed circular (linear) arcs.\up
\item
{\em Triarc curve} is composed of three circular arcs.
Curvatures $k_1,k_2,k_3$ of a {\em spiral triarc} form a monotone sequence.\up
\end{itemize}

\section{Notation and some properties of biarcs}

Let $[x(s),y(s)]$ be a planar curve, parametrized by the arc length~$s$.
The curvature at a point is defined by derivative
$k(s)=\tau'_s$, where
$\tau(s)=\arg(x'_s+\iu y'_s)$.
Curvature element
$\Kl{0}=\Brace{x_0,y_0,\tau_0,k_0}$
at the point $(x_0,y_0)$ includes the slope~$\tau_0$ of the tangent vector
$\nvec{\tau_0}=(\cos\tau_0,\,\sin\tau_0)$,
and curvature $k_0$, thus defining the directed circle of curvature at this point.

A spiral arc $\Arc{AB}$,
supported by the chord of the length $\abs{AB}=2c$,
is considered in the local coordinate system with point~$A$ 
moved to position $(-c,0)$, and $B$ to $(c,0)$.
Boundary curvature elements are\up
\Equa{K1K2c}{%
  \Kl{1} = \Brace{-c,0,\alpha,k_1},\quad
  \Kl{2} = \Brace{c,0,\beta,k_2};\quad
  \text{denote also}\quad
  \gamma=\frac{\alpha-\beta}2,\quad \omega=\frac{\alpha+\beta}2\,.
}

In
\RefFig[a]{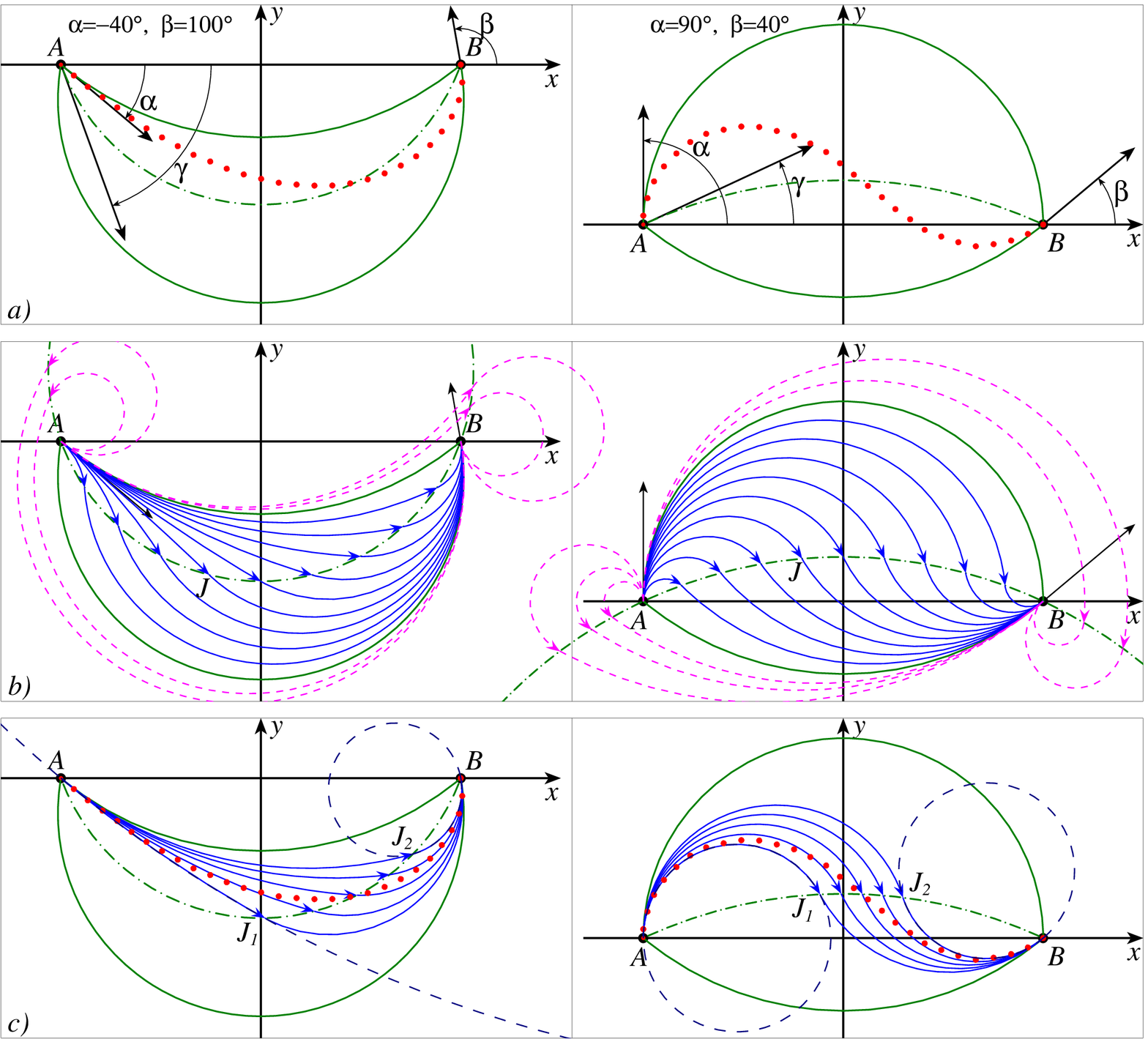}
two spiral arcs with boundary data \eqref{K1K2c} are shown by dotted lines;
the example to the left presents a convex curve, the rightward one shows
a curve with inflection.
Three circular arcs are traced from start point $A$ to end point $B$ of each curve.
One of them shares tangent $\nvec{\alpha}$ with the spiral at~$A$,
the other shares tangent $\nvec{\beta}$ at~$B$.
These two arcs form the {\em lens}.
The third circular arc, shown dotted-dashed, traced from~$A$ at the angle~$\gamma$,
is the {\em bisector of the lens\,}; $\omega$~is the angular half-width of the lens.

In~\cite{InvInv} the inversive invariant $Q$ of a pair of circles was proposed, equal to
$\sin^2\frac{\Psi}{2}$,
where $\Psi$ is intersection angle of two circles
(purely imaginary if $Q<0$).
For pair~\eqref{K1K2c}
\Equa{DefQ1}{%
    Q=(k_1c+\sin\alpha)(k_2c-\sin\beta)+\sin^2\omega.
}
According to \cite{PomiMain} (theorem 2), 
$Q<0$ is the necessary and sufficient condition for existence of a non-biarc spiral
with two-point G$^2$ Hermite data~\eqref{K1K2c}.
If $Q=0$, boundary circles of curvature are tangent, 
and the only possible spiral in this case is biarc.

\subsection{Short spirals}

We call a spiral arc $\Arc{AB}$ {\em short},
if it has no common points with its chord's complement to the infinite straight line
(possibly, intersecting the chord itself).
In~\cite{Sabitov} 
\quo{very short spirals} are considered,
namely, those, one-to-one projectable onto the chord.\up
\begin{itemize}
\item
Existence conditions for a short non-biarc spiral
with two-point G$^2$ Hermite data~\eqref{K1K2c} are: $Q<0$, and\up
\Equa{ABshort}{%
  \begin{array}{llll}
    \text{if~~}k_1<k_2{:}& -\pi<\alpha\le \pi,\quad& -\pi<\beta\le\pi,\quad&\alpha+\beta>0;\\  
    \text{if~~}k_1>k_2{:}& -\pi\le\alpha<\pi,& -\pi\le\beta<\pi,\quad&\alpha+\beta<0
  \end{array}
}
\cite[theorems\,1,\,3]{PomiShort}.
The choice between $\pm\pi$ is imposed by continuity of function $\tau(s)$,
whose values in internal points of a short spiral are in the interval $(-\pi;\pi)$.
E.\,g., spiral with \hbox{$\alpha=\beta=-\pi$} looks like
\setlength{\unitlength}{\ps}
\begin{picture}(40,9)\put(0,0){\Infigw{40\ps, bb=0 3 40 9}{AuxFigPiPi}}\end{picture}.\up
\item
Existence conditions for a short biarc
with two-point G$^2$ Hermite data~\eqref{K1K2c} require $Q=0$ and,
additionally to~\eqref{ABshort}, $\abs{\alpha+\beta}\ne2\pi$. 
This excludes cases $\alpha=\beta=\pm\pi$.\up
\item
Conditions \eqref{ABshort} include Vogt's theorem
{\rmfamily(see \cite{PomiShort}, theorem\,1 and references)}:
\Equa{Vogt}{%
     \sgn(k_2-k_1)=\sgn(\alpha+\beta).\up
}     
The theorem remains valid for long spirals,
if the values of boundary angles are uniquely precised as
$\alpha\to\alpha{+}2m\pi$, $\beta\to\beta{+}2n\pi$,
to become \quo{the angles, bearing their history}.
\item
Short spiral is enclosed into the lens {\rmfamily \cite[theorem\,2]{PomiShort}}.\up
\end{itemize}

\begin{figure}[t]
\centering%
\Pfig{1.\textwidth, bb=0 0 572 480}{Notation}{
(a)~Spiral arc (traced by dotted line), its boundary tangents, and the lens;
the left example is for $0\le k_1<k_2$, and the right one for $k_1<0<k_2$;
(b)~a family of short biarcs, filling the lens; examples of long biarcs are shown dashed;
(c)~boundary circles of curvature of the spiral arcs (dashed),
defining bilens.
}%
\end{figure}

\subsection{%
A family of biarcs with common end tangents}

In
\RefFig[b]{Notation}
the lens is filled with the family of short biarcs,
having common end tangents $\nvec{\alpha}$, $\nvec{\beta}$.
Join points~$J$ are marked by arrows.
Dashed curves show some examples of long biarcs.

Applying homothety with scale factor $c^{-1}$ causes transformations
\equa{
   A\to(-1,0),\quad B\to(1,0),\quad
   \alpha\to\alpha,\quad \beta\to \beta,\quad
   k_1\to a=k_1c,\quad k_2\to b=k_2c.
}
The values $a$ and $b$ become dimensionless curvatures of arcs $AJ$ and $JB$,
normalized to the chord length
$2c=2$.
The condition of tangency of these two arcs, $Q=0$, looks like
\Equa{Hyp-ab}{
    Q(a,b)=0,\quad\text{where}\quad
    Q(a,b) \Eqref{DefQ1} (a+\sin\alpha)(b-\sin\beta)+\sin^2\omega.
}
In the curvatures plane $(a;b)$ 
the curve $Q(a,b)=0$ is the hyperbola, shown in \RefFig{Hyperb}.
Its possible parametrizations $[a(p),\,b(p)]$ yield the parametrization
of the biarcs family under consideration. As in \cite{Bilens}, we accept \up
\Equa{BFamily}{%
   k_{AJ}=\frac{a(p)}c,\quad
   k_{JB}=\frac{b(p)}c,\quad\text{where}\quad
     a(p)=-\sin\alpha-\frac{\sin\omega}p,\quad
     b(p)=\sin\beta +  p\,\sin\omega.
}
Under such parametrization biarcs with $p>0$ are short
\cite{Bilens}, Proof of Property\,3).
Some properties of biarcs
$\Barc(p;\alpha,\beta)$
are mentioned here as functions of family parameter~$p$.

Join points $(X_J, Y_J)$ form a circular arc
(\cite[Property\,6]{Bilens})
\equa{
    X_J(p;\gamma)+\iu Y_J(p;\gamma)=c\,\frac{p^2{-}1 + 2\iu p\sin\gamma}{{p^2+2p\cos\gamma+1}}.
}
Its arc $0\le p\le\infty$
is the lens' bisector, the locus of join points of short biarcs.

The total turning of a short biarc is equal to $\beta-\alpha$
(and $\beta-\alpha\pm2\pi$ for a long one).
It is the sum of turnings $\theta_{1,2}$ of each subarc
\cite[Property\,9]{Bilens}:
\Equa{Rho12}{%
   \begin{array}{l}\theta_1(p)=\tauJ(p)-\alpha,\\ \theta_2(p)=\beta-\tauJ(p),\end{array}
   \quad
      \tauJ(p)={-2}\arctan\dfrac{p\sin\frac{\alpha}2+\sin\frac{\beta}2}{p\cos\frac{\alpha}2+\cos\frac{\beta}2}
              =-\omega+2\arctan\Skobki{\dfrac{1-p}{1+p}\cdot\tan\dfrac{\gamma}{2}}
}
(with no $\pm2\pi$ corrections for short biarcs).
The angle $\tauJ$ is the slope of tangent to a biarc at the join point.
The length $S(p)$ of a biarc is
\Equa{BiarcLen}{%
   S(p)=S_1(p)+S_2(p),
   \text{~~where~~}
   S_1(p)=c\frac{\theta_1(p)}{a(p)},\quad
   S_2(p)=c\frac{\theta_2(p)}{b(p)}.
}

{\em The length $S(p)$ of a short biarc is a strictly monotone function
of parameter~$p$, or constant,
if $\alpha=\beta$} \cite[Property\,11]{Bilens}.
Let us precise the type of monotonicity of $S(p)$ (not specified in \cite{Bilens})
by looking at limit cases
$p\to0$ and $p\to{+\infty}$.
If $p\to0$, the join point $J(p)$ of biarc $AJB$ tends to point~$A$,
arc $AJ$ vanishes,
and the biarc degenerates to the circular arc,
one of lens boundaries, the lower one under conditions of \RefFig{Notation}.
If $p\to\infty$, \ $J(p)\to B$, biarc $AJB$ degenerates to the second lens boundary.
Since\up
\equa{%
    S(0)=2c\,\frac{\beta}{\sin\beta},\qquad
    S(\infty)= 2c\,\frac{\alpha}{\sin\alpha},
}
$S(p)$ strictly decreases when $\abs{\alpha}<\abs{\beta}$.

If these two degenerated biarcs are taken into account,
there exists the unique biarc, passing through any point in plane,
except poles $A$ and $B$
\cite[Properties\,2,10]{Bilens}.

\subsection{Bilens}

Dashed circles in
\RefFig[c]{Notation}
are boundary circles of curvature~\eqref{K1K2c} of the spiral arc.
Biarc $AJ_1B=\Barc(p_1;\alpha,\beta)$ is chosen such that its first arc $AJ_1$ is coincident
with circle~$\Kl{1}$.
Arc $J_2B$ of biarc $AJ_2B=\Barc(p_2;\alpha,\beta)$ is coincident
with circle~$\Kl{2}$.
So, family parameters $p_{1,2}$ of these two biarcs are given 
by equalities  \hbox{$k_1c=a(p_1)$} and \hbox{$k_2c=b(p_2)$}:\up
\Equa{p1p2}{
   p_1=\frac{-\sin\omega}{k_1c+\sin\alpha},\qquad
   p_2=\frac{k_2c-\sin\beta}{\sin\omega}.
}
Normalized curvatures of two additional arcs, $J_1B$ and $AJ_2$, are
$b(p_1)$ and $a(p_2)$.

\begin{itemize}
\item
We call {\em bilens} the region, bounded by biarcs
$\Barc(p_1;\alpha,\beta)$ и $\Barc(p_2;\alpha,\beta)$.\up
\item
Bilens theorem \cite[Theorem\,1]{Bilens}:
{\em all short spirals with boundary curvature elements~\eqref{K1K2c},
are enclosed into the bilens}.\up
\item
{\em Bilens width}, defined as the maximal diameter of inscribed circles,
for convex spirals is given by Theorem~2 in~\cite{Bilens} as
\Equa{Bwidth}{%
   \Diam{}=\dfrac{4c(p_2-p_1)\sin\abs{\omega}}
           {P{+}\sqrt{P^2+4p_2(p_2-p_1)\, a(p_2)\, b(p_1)}},\quad
           \text{where~~}P=1+2p_2\cos\gamma+p_1p_2.
}
\end{itemize}

\RefFig{Hyperb}
illustrates the theorem in the curvature plane $(a;b)$ 
of normalized curvatures~\eqref{BFamily},
end tangents $\nvec{\alpha}$ and $\nvec{\beta}$ being fixed.
Possible values of boundary curvatures $(a,b)$ are defined by inequality
$Q(a,b)\le 0$~\eqref{Hyp-ab}.
For a spiral with increasing curvature this refers to the convex region,
bounded by the left (upper) branch of the hyperbola, located in half-plane
$a<b$.
This is also the branch $0<p<\infty$ of the curve
$[a(p),\,b(p)]$~\eqref{BFamily}.
Its location with respect to asymptotes of the hyperbola yields inequalities
(\cite{PomiMain}, cor.\,2.1)\up
\Equa{Asymp}{%
   a<b \So    a < -\sin\alpha,\quad  b>\sin\beta.
}

Point
$K=(k_1c,k_2c)=(a_1,b_2)$
corresponds to given boundary curvatures
of the spiral arc,
and defines parameters $p_{1,2}$ of the bilens~\eqref{p1p2}.
By projecting point $K$ onto hyperbola, we obtain two points, $J_1=(a_1,b_1)$, and $J_2=(a_2,b_2)$.
Their coordinates are equal to curvatures of arcs $AJ_1,\,J_1B$ and $AJ_2,\,J_2B$,
bounding the bilens.
Points of arc $J_1J_2$ of the hyperbola are images of biarcs, filling the bilens:
$J_1J_2=[a(p),\,b(p)]$, $p_1<p<p_2$. 

\begin{figure}[t]
\centering%
\Pfig{.92\textwidth, bb=0 0 508 200}{Hyperb}
{Hyperbola $Q(a,b)=0$ \ \eqref{Hyp-ab} for two sets $\Brace{\alpha,\beta}$, same as in
\RefFig{Notation}
}%
\end{figure}

In terms of \RefFig{Hyperb}, bilens theorem sounds as follows:
{\em short spiral arcs with boundary curvatures~\eqref{K1K2c},
belonging to curvilinear triangle $KJ_1J_2$,
are inside the bilens}\,.
If $p_1\to 0$
($k_1\to\pm\infty$)
and $p_2\to \infty$ ($k_2\to\mp\infty$),
triangle $KJ_1J_2$ transforms to infinite region to the left of the left branch
of the hyperbola \eqref{Hyp-ab}, or, in the case of decreasing curvature,
to the right of its right branch;
the bilens transforms to the lens.

\subsection{Convex biarcs}

Now consider convex biarcs,
in particular, the limitary case of convexity, namely, biarc
$\Barc(\bar{p};\alpha,\beta)$,
whose one subarc has zero curvature
($a(\bar{p}) = 0$ or $b(\bar{p}) = 0$).

A convex biarc must be short ($p>0$), because a convex curve
cannot have the third common point with the complement of its chord to the infinite straight line
(X-axis).
And curvatures $a,b$ cannot have opposite signs.
Inequality $a(p)\cdot b(p) \ge 0$ at $p>0$ is solved as:\up
\begin{subequations}\label{PSstar}
\Equa{PSstar1}{%
  \aligned
     &\text{if~~}\abs{\alpha}<\abs{\beta}{:}\;&
      \bar{p}&{}\le p,\quad&
      &\text{where~~}\bar{p}=-\dfrac{\sin\omega}{\sin\alpha}>0\quad&
      &\Brack{\,a(\bar{p}) = 0 \text{~~in this case}};
     \\
     &\text{if~~}\abs{\alpha}>\abs{\beta}{:}\;&
     0&{}< p \le \bar{p},\quad&
     &\text{where~~}\bar{p}= -\dfrac{\sin\beta}{\sin\omega}&
      &\Brack{\,b(\bar{p}) = 0\,}.
  \endaligned
}
  
The case  $a(\bar{p})=0$ is shown in \RefFig{S1S2} as biarc
$AJ_0B = \Barc(\bar{p})$.
The length of the straight segment $AJ_0$ can be defined
as the limit of $S_1(p)$ \eqref{BiarcLen} when $a(p)\to 0$ 
($p\to \bar{p}=-\frac{\sin\omega}{\sin\alpha}$):
\equa{
    \abs{AJ_0}=\lim\limits_{a(p)\to 0}  c\frac{\theta_1(p)}{a(p)}
              = -2c\dfrac{\sin\omega}{\sin\gamma}.
}
Turning angle $\theta_2(\bar{p})$ of the second arc, $J_0B$, is
$\beta-\alpha=-2\gamma$,
and its curvature and length are
\equa{
   b(\bar{p})=\sin\beta +  \bar{p}\sin\omega=\sin\beta-\frac{\sin^2\omega}{\sin\alpha}=-\frac{\sin^2\gamma}{\sin\alpha},
   \qquad
   S_2(\bar{p})=c\frac{\theta_2(\bar{p})}{b(\bar{p})}=2c\frac{\gamma\sin\alpha}{\sin^2\gamma}.
}
The total length of biarc $\Barc(\bar{p})$ for this case is given
as the first case in~\eqref{PSstar2}.

For the case $\abs{\alpha}>\abs{\beta}$ we obtain similarly 
$a(\bar{p})=\frac{\sin^2\gamma}{\sin\beta}$,
$S_1(\bar{p})=-2c\frac{\gamma\sin\beta}{\sin^2\gamma}$,
and, for the segment of zero curvature, $S_2(\bar{p})=2c\frac{\sin\omega}{\sin\gamma}$.
So,
\Equa{PSstar2}{%
   S(\bar{p}) = \left\{
   \begin{array}{lll}
      \dfrac{2c}{\sin\gamma}\Skobki{\dfrac{\gamma\sin\alpha}{\sin\gamma}-\sin\omega}\,,\quad&
      \text{if~~~} \abs{\alpha}<\abs{\beta};\quad\\[2ex]
      \dfrac{2c}{\sin\gamma}\Skobki{\sin\omega-\dfrac{\gamma\sin\beta}{\sin\gamma}}\,,\quad&
      \text{if~~~}\abs{\alpha}>\abs{\beta}\quad
   \end{array} 
   \right.\quad
   \Skobki{
   \begin{array}{l}
      \alpha=\omega+\gamma\\
      \beta=\omega-\gamma.
   \end{array} 
   }\,.
}
\end{subequations}

\section{Generalization of theorem 1 \cite[Sabitov, Slovesnov]{Sabitov}}

\begin{subequations}\label{eqTh1}
\begin{thm}\label{ThSabitov}
Let~$\Gamma$ be a convex spiral of length $L$
with increasing curvature~$k(s)$, and
\Equa{eqTh1a}{
  k_1=k(0),\quad k_2=k(L):\qquad 0\le k_1 < k_2 \text{~~or~~} k_1 < k_2\le0.
}  
There exists a unique biarc~$\Gamma_0$, approximation of~$\Gamma$,
with the same length, same end points, and same end tangents.
Curvatures $   q_1$~and~$q_2$ of two arcs of biarc~$\Gamma_0$
obey inequalities
\Equa{eqTh1b}{
    k_1\le q_1<q_2\le k_2,
}
with equalities arising if and only if~$\Gamma$ itself is biarc $(\Gamma_0=\Gamma)$.

If curvature decreases, inequalities \eqref{eqTh1} are replaced by the opposite ones.
\end{thm}

The statement of the theorem includes both the statement of Theorem\,1 from~\cite{Sabitov},
and its strengthening for $k_1=0$ or $k_2=0$  \cite[p.\,5]{Sabitov}.
Additional features are:
\begin{itemize}
\item 
The proof uses results, previously proven for $C^1$-continuous curves.
Therefore we do not require curvature continuity,
accept {\em non}strict monotonicity, in particular, piecewise constancy.\up
\item
Uniqueness of the solution is stated.
\item
Restriction on the total turning,
$\left(\theta=\int_0^L k(s)ds<\frac{\pi}2\right)$,
is weakened: convexity of the curve is sufficient,
which admits the turning angle as close to~$\pm2\pi$ as one pleases.
\end{itemize}

Note that such values do not necessarily yield a bad approximation:
{\em
its precision is the width of the bilens} \eqref{Bwidth},
which could be arbitrarily small
even if $\abs{\theta}\lessapprox 2\pi$.

\begin{proof}[{\textbf{\proofname}}]
The situation is illustrated by \RefFig[c]{Notation}, whose left fragment
shows one of two options of \eqref{eqTh1a}: $0\le k_1<k_2$,
i.\,e. a convex spiral with non-negative curvature.
Spiral~$\Gamma$ is shown by dotted line.

The case when $\Gamma$ is a biarc, and $\Gamma_0=\Gamma$, is trivial;
the uniqueness of the approximation results from monotonicity
of function~$S(p)$~\eqref{BiarcLen}
for short biarcs.

In the non-trivial case we have to prove strict inequalities~\eqref{eqTh1b}.
Consider them in the plane $(a;b)$ of normalized curvatures as
\Equa{eqTh1c}{
    a_1<a_0<b_0<b_2,\text{~~where~~}\
    a_1=k_1c,\quad a_0=q_1c,\quad b_0=q_2c,\quad b_2=k_2c.
}
The left side of \RefFig{Hyperb} corresponds to the left side of \RefFig[c]{Notation}
in the sense of identical boundary angles $\alpha,\beta$,
which, in \RefFig{Notation}, define the lens, 
and in \RefFig{Hyperb} define asymptotes of the hyperbola.
From two inequalities~\eqref{eqTh1a}, the first option
($0\le k_1 < k_2$) is drawn in the figure:
region $KJ_1J_2$ is located in the quadrant $a\ge0$, $b\ge0$ of non-negative curvatures;
and, more precisely, in the octant $a<b$ of increasing curvature.

According to the bilens theorem, curve~$\Gamma$ in enclosed by the bilens.
Involved curves are convex,
curve~$\Gamma$ surrounds one of bilens boundaries, biarc $AJ_2B$,
and the second boundary, biarc $AJ_1B$, surrounds curve~$\Gamma$.
By theorem~3 from
\cite[p.\,411]{Alexandrov},
\equa{
     S(p_1) = L_{AJ_1B}<L_\Gamma < L_{AJ_2B} = S(p_2).
}
Because lengths $S(p)$ of biarcs, filling the bilens,
vary monotonously,
solution~$p_0$ of the equation \hbox{$S(p_0)=L_\Gamma$} exists, is unique, enclosed in the range $p_1<p_0<p_2$,
and yields the sought for biarc. 
The image of this biarc is one of points
$(a_0,\,b_0)=(a(p_0),\,b(p_0))$ of arc $J_1J_2$ of the hyperbola,
completely located in the octant $0\le a<b$,
and inequalities~\eqref{eqTh1c} and~\eqref{eqTh1b} hold.
\end{proof}
\end{subequations}

\renewcommand{\tmp}[1]{#1^{\text{\cite{Sabitov}}}}
In terms of article~\cite{Sabitov} notation $\alpha$ designates the total turning of the spiral arc.
Eq.\,$\tmp{(4.4)}$ for the numerical solution,
rewritten with replacement $\tmp\alpha\to\theta$, looks like
\equa{%
  (R-r)\times\arccos\left(\frac{R-y-r\cos\theta}{R-r}\right)+r\theta=L,
  \text{~~~where~~~}
  R=\frac{x^2+y^2+2r(-x\sin\theta+y\cos\theta)}{2(y-r+r\cos\theta)}.
}
Further replacements bring this equation to the notation and coordinate system
of this article:
\equa{%
  \theta\to\beta-\alpha=\theta_1(p)+\theta_2(p),\quad
  \tmp{R}\to\frac{c}{a(p)},\quad \tmp{r}\to\frac{c}{b(p)},\quad
  \tmp{x}\to2c\cos\alpha,\quad \tmp{y}\to{-2c}\sin\alpha
}
(assuming positive curvature). Arccosine transforms to $\theta_1(p)$,
and Eq.\,$\tmp{(4.4)}$ becomes equivalent to the equation $S(p)=L_\Gamma$.

\section{On lengths of convex spiral arcs}

The converse of theorem~1 is hardly of interest in the view of approximation problem.
Nevertheless, there are situations, when it becomes useful.
E.\,g., theorem\,1 finds for a convex spiral
\begin{subequations}
\Equa{ClassG}{
   \Gamma(c,\alpha,\beta,k_1,k_2),\quad 
    k_1=-\frac1c\Skobki{\sin\alpha+\frac{\sin\omega}{p_1}},\quad
    k_2=\frac1c\Skobki{\sin\beta+p_2\sin\omega},
}
the unique biarc $\Barc(p_0;\alpha,\beta)$ from the subfamily of biarcs
\Equa{FamilyB}{
      \Barc(p;\alpha,\beta),\quad p_1< p < p_2,
}
\end{subequations}
enclosed by the bilens. Monotonicity of $S(p)$ results in
unimprovable
inequalities
$S(p_1)\lessgtr S(p_0)\lessgtr S(p_2)$
for its lengths $S(p_0)=L_\Gamma$.
These inequalities could be extended onto the whole space of curves~$\Gamma$.
Theorem~2, the converse of theorem~1, being valid,
thus associating any biarc~\eqref{FamilyB}
with at least one curve of class~\eqref{ClassG},
they could be extended as {\em unimprovable} inequalities.
Theorem~2 legitimates also the modelling scheme, described in the next section.

\begin{thm}\label{ThInv}
Let $\Gamma_0$ be a convex biarc of length $L_0$, whose two curvatures,
\hbox{$q_1$ and $q_2$,}
are within the range $[k_1;\,k_2]$, such that $0\le k_1<q_1<q_2<k_2$. 
Then a convex spiral $\Gamma$ exists, of the same length~$L_\Gamma=L_0$, 
with the same endpoints and endtangents as~$\Gamma_0$,
and with boundary curvatures
$k(0)=k_1$, $k\Skobki{L_\Gamma}=k_2$.
\end{thm}
\begin{proof}[{\textbf{\proofname}}]
\RefFig[a]{AT1T2B} shows biarc $\Gamma_0$ as curve
$AJB=\Barc(p;\,c,\alpha,\beta)$.
From given curvatures $k_{1,2}$ bilens parameters $p_{1,2}$ are defined \eqref{p1p2},
and bilens is constructed, bounded by biarcs
$AJ_1B=\Barc(p_1;\,c,\alpha,\beta)$ and
$AJ_2B=\Barc(p_2;\,c,\alpha,\beta)$.
For biarc  $\Gamma_0$ inequalities~\eqref{ABshort} hold, and
\equa{%
    Q_0\Eqref{DefQ1}(q_1c+\sin\alpha)(q_2c-\sin\beta)+\sin^2\omega=0.
}
If spiral $\Gamma$ exists, $Q_\Gamma < 0$ should be satisfied, where
\equa{%
    Q_\Gamma= (k_{1}c+\sin\alpha)(k_{2}c-\sin\beta)+\sin^2\omega.
}
In the case of increasing curvature,  $k_1<q_1<q_2<k_2$,
\equa{%
    k_{1}c+\sin\alpha<q_1c+\sin\alpha<0,\quad
    k_{2}c-\sin\beta>q_2c-\sin\beta>0.
}
Comparisons to zero result from \eqref{Asymp},
and yield $Q_\Gamma<0$.
Together with inequalities~\eqref{ABshort}, inherited from biarc $AJB$,
this constitutes the necessary and sufficient condition of the existence of spiral~$\Gamma$
with required boundary data
\cite[theorem\,3]{PomiShort}.

\begin{figure}[t]
\centering%
\Pfig{.96\textwidth}{AT1T2B}{%
(a)~Example of a triarc (curve $AMNB$), inscribed into the bilens: 
arcs $AM$ and $NB$ are partially coincident with bilens boundaries $AJ_1$ and $J_2B$;
(b)~the family of inscribed triarcs;
(c) M\"{o}bius map of configuration~(b), bringing boundary circles of curvature to concentricity.
}%
\end{figure}

To satisfy also the requirement $L_\Gamma=L_0$, we construct a family of triarcs,
inscribed into the bilens
(\RefFig[b]{AT1T2B}).
The construction becomes quite simple, if we use M\"{o}bius map to transform
boundary circles of curvature to a concentric pair\footnotemark.
\footnotetext{%
Two maps exist, preserving points $A$ and $B$ ($z=\pm1$) intact.
One of them transforms increasing curvature $k_1<k_2$ to increasing negative,
$a<b<0$, $\frac{b}{a}=\varkappa^{-}<1$,
the other to increasing positive, $0<a<b$, $\frac{b}{a}=\varkappa^{+}>1$;
and $\varkappa^{\pm}=\Skobki{\sqrt{1-Q}\pm\sqrt{-Q}}^2$.
Maps look like\up
\equa{
  w(z)=\dfrac{z_0+z}{1+z_0 z},\quad\text{where}\quad
    z_0=\dfrac{r_0\Exp{\iu\lambda_0}-1}{r_0\Exp{\iu\lambda_0}+1},\quad
    r_0=\sqrt{\frac{\varkappa}{p_1p_2}},\quad
    \lambda_0=\pi-\gamma+\arctan\Skobki{\frac{\varkappa-1}{\varkappa+1}\cot\omega}. 
}
}
This is possible due to condition $Q<0$ (circles do not intersect).
Constructing in the case of concentricity is simplified by the fact that
any arc, joining two circles, is the semicircle,
and all of them have the same curvature;
the whole construction can be easily described in polar coordinates
with the pole~$O$ in the common center of two circles.

M\"{o}bius map preserves the values of $Q$, $\omega$,
and the very fact (and type) of spirality.
Tangencies, used to construct the bilens, are preserved,
as well as the order of tangency of two curves:
circles of curvature of an arc remain such after transformation.
{\em Non-invariant} are turning angles of curves
(in particular, the property on an arc to be a semicircle),
shortness and convexity of a curve;
but these aspects are not used in the further reasoning.

In \RefFig[c]{AT1T2B}
semicircles $AJ'_2$ and $J'_1B$
are images of arcs $AJ_2$ and $J_1B$;
they join smoothly boundary circles of curvature, now concentric.
Together with arcs $AJ_1'$ and $J_2'B$ they bound the region,
the image of the original bilens.
The family of semicircles $M'N'$, $M'\in AJ'_1$,
fills the region everywhere densely.
When the polar ray $OM'$ sweeps out sector $AOJ'_1$,
ray $ON'$  sweeps out the opposite sector $J_2'OB$.
We obtain the family of inscribed triarcs $AM'N'B$ such that:

$\bullet$~arc $AM'$ is partially coincident with the circle of curvature at the startpoint;

$\bullet$~it is continued by the transition curve, semicircle $M'N'$;

$\bullet$~the third arc, $N'B$, assures the required curvature at the endpoint;

$\bullet$~the three curvatures form monotone sequence.\\
The backward transformation provides the family of triarcs $AMNB$,
everywhere densely inscribed into the bilens.
As $M_1$ moves to $A$,
$N'$ moves to $J'_2$, and $M\to A$, $N\to J_2$:
the first arc of the triarc vanishes,
the triarc degenerates to biarc $AJ_2B$ ($M_2\to J_2$).
Similarly, as $N'\to B$, the third arc of the triarc vanishes,
triarc $AMNB$ degenerates to biarc $AJ_1B$.
{\em The lengths of inscribed triarcs vary continuously in the same range,
wherein the lengths of inscribed biarcs vary continuously and monotonically.}
Therefore the sought for spiral $\Gamma$ of length $L_0$ exists,
at least as a triarc curve.
\end{proof}

\begin{cor}
The length $L_\Gamma$ of a convex spiral arc $\Gamma$ with boundary data~\eqref{K1K2c}
obeys unimprovable inequalities\up
\Equa{Lengths}{
  \aligned
   &\text{if~~}\abs{\alpha}<\abs{\beta}{:}\;&
   S(\infty)< S(p_2) \le{}&L_\Gamma\le S(p_1)\le S(\bar{p})\;&  
   &\Brack{\begin{array}{lcl}
         0\le k_1<k_2, & & \beta> -\alpha>0;\\ 
         0\ge k_1>k_2, & & \beta< -\alpha<0; 
     \end{array}}.
    \\[1.5ex]
    &\text{if~~}\abs{\alpha}>\abs{\beta}{:}\;&
    S(0)<S(p_1)\le{}&L_\Gamma\le S(p_2)\le S(\bar{p})\;&
    &\Brack{\begin{array}{lcl}
         k_1<k_2\le0, & & \alpha> -\beta>0;\\ 
         k_1>k_2\ge0, & & \alpha< -\beta<0; 
     \end{array}}.
  \endaligned
}
Here the inner inequalities account for boundary curvatures;
equalities arise if and only if $\Gamma$ is biarc $(p_1=p_2)$.
The outer inequalities  account for boundary angles only;
the equality case arises if $\Gamma$ is the biarc with straight line segment $(p_1=p_2=\bar{p})$.
\end{cor}

\begin{proof}[{\textbf{\proofname}}]
This is an immediate corollary of theorems 1 and 2.
Below we simply comment additional [bracketed] inequalities, accompanying cases
$\abs{\alpha}\gtrless\abs{\beta}$.
\begin{itemize}
\item
If curvature is non-positive and increasing, $k_1<k_2\le0$, we have
$\sin\beta<k_2c\le 0$ \eqref{Asymp},
i.\,e. $\sin\beta<0$,  $\beta<0$.
Together with Vogt's theorem, $\alpha+\beta>0$ \eqref{Vogt}, this yields
$\alpha> -\beta>0$.
\item
In the case $0\le k_1<k_2$ we have $0\le k_1c<-\sin\alpha$, and $\alpha+\beta>0$, i.\,e. $\sin\alpha<0$, 
\hbox{$\alpha<0$}, and therefore \hbox{$\beta>-\alpha>0$}.
\item
The cases of decreasing curvature can be brought to the above ones by the symmetry about X-axis,
which looks like the sign changes for $\alpha,\beta,k_1,k_2$.
\vspace{-1.5\baselineskip}
\end{itemize}
\end{proof}

\begin{figure}[t]
\centering%
\Pfig{1\textwidth, bb=0 0 540 200}{S1S2}{
Illustration to inequalities \eqref{Lengths} 
}%
\end{figure}

\RefFig{S1S2} illustrates inequalities \eqref{Lengths}.
The case $\abs{\alpha}<\abs{\beta}$ with increasing curvature is shown.
Because $\beta=\pi$ is chosen, one of lens boundaries has transformed into X-axis
(with chord $AB$ cut off),
and the horizontal asymptote of hyperbola $Q(a,b)=0$ 
(in the right fragment) became the axis $b=0$.
One of non-convex biarcs is also shown ($AJB$ with $k_1<0<k_2$),
related to the part of the hyperbola in the second quadrant ($a<0$, $b>0$).
Since only convex curves are considered,
we are interested in the subregion of
curvatures $Q(a,b)\le0$, falling into quadrants with $ab\ge 0$,
i.\,e. the first or the thirds ones.
In \RefFig{S1S2} this is shaded subregion $Q(a,b)\le0$, $0\le a<b$ in the first quadrant.

Biarc $AJ_0B$ is $\Barc(\bar{p})$, because $k({AJ_0})=0$.
The subfamily of convex biarcs, $\Barc(p)$, $\bar{p}\le p\le \infty$ \eqref{PSstar1},
is bounded by this biarc and by the circular arc $AMB=\Barc(\infty)$, lens' boundary.
The lengths of these two curves form limits to the length $L_\Gamma$ of an arbitrary
convex spiral with chord $\abs{AB}=2c$ and boundary tangents
$\nvec{\alpha}$ and $\nvec{\beta}$: $S(\bar{p})\ge L_\Gamma>S(\infty)$. 
The second inequality is strict, because curve $\Barc(\infty)$ is no more
a curve with given tangents, although can be as close as one pleases to such one.

If boundary curvatures of a spiral arc are known,
bilens $p\in[\bar{p};\infty)$ (or  $p\in(0;\bar{p}]$) narrows down to $p\in(p_1;p_2)$, 
and inequalities \eqref{Lengths} become correspondingly constricted.
\smallskip

These inequalities, having rather simple geometric construction,
have rather lengthy algebraic form.
Below we put together the sequence of required calculations. 
First, apply the symmetry about one or both coordinate axes,
in order to bring any of four possibilities~\eqref{Lengths}\up
\equa{%
  \setlength{\unitlength}{\ps}
  \begin{picture}(330,28)\put(0,0){\Infigw{330\ps}{Symm4}}\end{picture}
}
to the first one, with increasing non-negative curvature.
Angles $\alpha,\beta$ become $\tau_1$ and $\tau_2$, normalized curvatures
$k_1c,k_2c$ become $a_1$~and~$b_2$:
\equa{
  \aligned
  &\tau_1=-\min\Skobki{\abs{\alpha},\abs{\beta}},\quad&
  &\tau_2=\max\Skobki{\abs{\alpha},\abs{\beta}}\qquad&
  &\Brack{0<-\tau_1<\tau_2\le\pi};\\
  &a_1=c\cdot\min\Skobki{\abs{k_1},\abs{k_2}},\quad&
  &b_2=c\cdot\max\Skobki{\abs{k_1},\abs{k_2}}\quad&
  &\Brack{ 0\le a_1 < b_2};\\
  &b_1=\frac{a_1\sin\tau_2-\sin^2\gamma}{a_1+\sin\tau_1},\quad&
  &a_2=-\frac{b_2\sin\tau_1+\sin^2\gamma}{b_2-\sin\tau_2},\quad&
  &\omega=\frac{\tau_1+\tau_2}2,\quad 
  \gamma=\frac{\tau_1-\tau_2}2.
  \endaligned
}
The third line yields additional curvatures, 
$b_1=b(p_1)$ of arc $J_1B$, and
$a_2=a(p_2)$ of arc $AJ_2$,
expressed through given curvatures $a_1$~and~$b_2$.
Turning angles of arcs, $\theta(AJ_1)=2\xi_1$ and $\theta(J_2B)=2\xi_2$,
are in the range $[0;\,2\pi)$,
and therefore can be exactly got as \ $2\arccot\Skobki{\cot\frac{\theta}{2}}$.
We express them through known curvatures as
$2\xi_1=\theta_1(p(a_1))$ and $2\xi_2=\theta_2(p(b_2))$~\eqref{Rho12}:
\equa{
  \xi_1=\arccot\frac{a_1\cos\omega+\sin\gamma}{-a_1\sin\omega},\qquad
  \xi_2=\arccot\frac{b_2\cos\omega+\sin\gamma}{b_2\sin\omega}.
}
Turning angles of complementary arcs can be defined from the total turning $\tau_2-\tau_1=-2\gamma$.
Inequalities~\eqref{Lengths} are rewritten below in terms of curve length to chord length ratio:
\equa{
   \frac{\tau_1}{\sin\tau_1}  
   < -\frac{\gamma+\xi_2}{a_2}+\frac{\xi_2}{b_2}       
   \le \frac{L_\Gamma}{2c}
   \le \frac{\xi_1}{a_1}-\frac{\gamma+\xi_1}{b_1}       
   \le \frac{\gamma\sin\tau_1-\sin\omega\sin\gamma}{\sin^2\gamma}. 
}
Indeterminacy at $a_1=0$ is evaluated by replacing the last inequality of this chain by equality.

\section{Investigation of properties of planar curves by modelling}

In this section we are going to demonstrate the scheme of curves modelling,
aimed to explore properties of curves under some  particular constraints,
to formulate or verify some hypotheses or preliminary propositions.
That's why the conclusions, suggested by the below examples,
are left as hypotheses, without an attempt to prove them.

The principle of modelling is based on
theorems 1 and 2, allowing us to replace the analysis in infinite-dimensional space of monotonic functions
$k(s)$ ($0\le s\le L$, $k_1\le k(s)\le k_2$) 
by enumeration in the three-parametric family 
$\bar{k}(s;\,C_1,C_2,C_3)$
of piecewise constant functions, namely
\settowidth{\tmplength}{$L_2$}       
\Equa{Model1}{
     \bar{k}(s;\,q_1,q_2,l_1) = 
     \begin{cases}
        q_1, & \text{~~if~~}\makebox[\tmplength][r]{$0$}\le s\le l_1,\\
        q_2, & \text{~~if~~}\makebox[\tmplength][r]{$l_1$}< s\le L
     \end{cases}\qquad
     \left[\begin{array}{l}0<l_1<L,\\
     k_1 < q_1 < q_2 < k_2\end{array}\right]\,.
}
Such enumeration 
can be coded by 3~embedded loops with some reasonable step over each of 3 parameters.

\subsection{Modelling of spirals: location of endpoints}

In~\cite{Sabitov} possible positions of the endpoint of a spiral arc
with prescribed end curvatures $k(0)=k_1$, $k(L)=k_2$, and length~$L$
are of interest.
In \RefFig{Example1} five examples of such arcs are traced.
The left column shows plots of curvature $k(s)$, monotone increasing from $k_1$ to $k_2$.
There are also drawn two-level plots of curvatures $\bar{k}(s)$ of biarcs,
approximating every spiral arc in terms of Theorem\,1.
Biarcs themselves are drawn in the second column,
together with some points of the original (approximated) curve.
Dashed arc is the circle of curvature $\Brace{0,0,0,k_1}$ at the startpoint. 

\begin{figure}[t]
\centering%
\Pfig{1\textwidth, bb=0 0 500 420}{Example1}{
Modelling of location of endpoints of a spiral arcs (case $0\le k_1<k_2$)
}%
\end{figure}

\begin{figure}[t]
\centering%
\Pfig{1\textwidth, bb=0 0 566 330}{Example2}{
Modelling of location of endpoints of a spiral arcs, the case $k_1>k_2\ge 0$.
}%
\end{figure}

Let complex number $Z_1(k;l)$ denotes the endpoint of the circular arc, 
whose curvature and length are $k$ and $l$,
traced from the coordinate origin along the X-axis:
\equa{
      Z_1(k;l)=\Int{0}{l}{\Exp{\iu ks}}{s}=\frac{\iu}{k}\Skobki{1-\Exp{\iu kl}}
       =\frac2k\sin\frac{kl}2 \Exp{\frac{\iu kl}2}\qquad
       \Brack{Z_1(0;l)=\lim\limits_{k\to 0}Z_1(k;l)=l}.
}
Let $Z_2(q_1,l_1;q_2,l_2)$ denotes the endpoint of the biarc,
whose curvatures and lengths are $q_1, l_1$, and  $q_2,l_2$:
\equa{%
      Z_2(q_1,l_1;q_2,l_2)=Z_1(q_1,l_1)+\Exp{\iu q_1 l_1}Z_1(q_2,l_2).
}

For every curve in model~\eqref{Model1} we calculate the endpoint as
$Z_2(q_1,l_1;\,q_2,L-l_1)$.
The result of such procedure is shown in the right side of \RefFig{Example1}
as the pointset, bounded by curves $\Gamma_1$ and~$\Gamma_2$.
As the hypotheses for bounds,
the limitary cases of model~\eqref{Model1} were thought of,
namely:
\begin{itemize}
\item
Vanishing of one of two arcs, i.\,e. either $l_1=0$, or $l_1=L$:
biarc~\eqref{Model1} degenerates into the circular arc
of curvature $q$ and length~$L$.
With varying~$q$, the endpoints of these arcs trace parametric curve
$\Gamma_1(q)$~\eqref{Gamma12}.
\item
Biarcs of the total length $L$ with curvatures,
taking limit values $q_1=k_1$, $q_2=k_2$,
the join point being varied.
Their endpoints trace curve~$\Gamma_2(t)$:
\Equa{Gamma12}{
   \begin{array}{lll}
      \Gamma_1{:}\quad& x(q)+\iu y(q)=Z_1(q;L),\quad& k_1\le q\le k_2;\\
      \Gamma_2{:}\quad& x(t)+\iu y(t)=Z_2(k_1,L{-}t;\,k_2,t),\quad& 0\le t\le L.
   \end{array}
}
\end{itemize}

In \RefFig{Example1} curves $\Gamma_{1,2}$ are extended beyond the parameter ranges,
specified in~\eqref{Gamma12}.
Curve $\Gamma_1$ is known as {\em cochleoid} \cite[p.\,230]{Savelov}.
Its polar equation is $p(\varphi)=L\frac{\sin\varphi}{\varphi}$.

Under conditions of \RefFig{Example1} ($0<k_1<k_2$), curve~$\Gamma_{2}$ is hypocycloid.
To bring it to the canonical position, one should move the coordinate origin to the point $\Skobki{0,k_1^{-1}}$,
and apply rotation by the angle $k_1L-{\pi}/2$.

The output of modelling for the case of curvature, decreasing in the range $k_1>k_2\ge 0$,
is shown in \RefFig{Example2} in the same manner.
Bound~$\Gamma_{2}$ becomes epicycloid, or, if $k_2=0$,
involute of initial circle of curvature.

Some more examples of such pointsets and their bounds are shown in \RefFig{Example12}.
In the leftmost picture bound~$\Gamma_2$ becomes cycloid ($k_1=0$), or straight line segment ($k_2=2k_1$).
The last picture shows additionally subsets of endpoints,
obtained by modelling with fixed value of turning angle
\equa{%
   \theta=\int_0^L k(s)ds=q_1l_1+q_2(L-l_1),
}
which is not affected by the approximation.
Boundaries of these subsets are guessed as curves\up
\equa{
  \aligned
   &\text{} && Z_2(q(t),L{-}t;\,k_2,t),
   &&\text{where}&& q(t)(L-t)+k_2t=\theta,&  0\le {}&t\le\frac{\theta-k_1L}{k_2-k_1},\\
   &\text{and~~} && Z_2(k_1,L{-}t;\,q(t),t),
   &&\text{where}&& k_1(L-t)+q(t)t=\theta,\quad&  \frac{\theta-k_1L}{k_2-k_1}\le {}&t\le L.
  \endaligned
}

\begin{figure}[t]
\Pfig{\textwidth}{Example12}{
Other examples, with curvature decreasing, increasing, positive, negative. 
}%
\end{figure}

\subsection{Modelling of closeness of ovals}

\enlargethispage{1\baselineskip}
{\em Oval} is usually meant a closed convex curve.
Additionally, we assume it to have the minimal number of verticez,
i.\,e., exactly four, in virtue of the Four-vertex theorem.
An example is given by curve $V_0V_1V_2V_3V_0$ in \RefFig[a]{Oval}.

Consider the possibility of modelling curves with only one vertex, namely,
with curvature $k(s)$, increasing from $k(0)=k_1$ to $k(L_1)=k_2$ on an arc of length $L_1$, 
and then decreasing to $k(L)=k_3$ on an arc of length $L_2=L-L_1$.
We need an analogue of model~\eqref{Model1},
as before, piecewise constant, but with $2\times 3=6$ free parameters.
Fixation of total turning~$\theta$ of a curve decrements the number of degrees of freedom
(and embedded loops) to five.
An example of modelling of a set of endpoints for a curve with one vertex and fixed turning
is shown in \RefFig[b]{Oval}.

As a candidate to oval, consider a curve with periodic curvature \hbox{$k(s)\ge 0$},
having four extrema within the period $s_4=L_1{+}L_2{+}L_3{+}L_4$:\up
\equa{%
   k_1=k(0)=k(s_4),\quad k_2=k(s_1),\quad k_3=k(s_2),\quad k_4=k(s_3),\quad
   \text{and}\quad k_1<k_2>k_3<k_4>k_1.
}
Several versions of $k(s)$ with prescribed~$s_i,k_i$ are shown below:\up
\Equa{Candidate}{%
\parbox{400pt}{\Infigw{400pt}{Oval-ks1}}
}
Denote curves, generated by such curvature function, as $V_0V_1V_2V_3V_4$.
The curve is oval if:\up 
\equa{
   \text{(a)~}\textstyle\int\limits_{0}^{s_4}{k(s)}\mathrm{d}s=2\pi;\quad \text{(b)~}V_0=V_4.
}\up\up

Let $\mu=\int\limits_{0}^{s_2}{k(s)}\mathrm{d}s$
and $\nu=\int\limits_{s_1}^{s_3}{k(s)}\mathrm{d}s$ \ be the turning angles between two opposite verticez,
those of minimal curvature ($\mu$, arc $V_0V_1V_2$),  
and those of maximal curvature ($\nu$, arc $V_1V_2V_3$).
For~$\mu$, and for turning $2\pi-\mu$ on the complementary arc $V_2V_3V_0$ we have natural restrictions
(here $L_1=s_1$, $L_2=s_2-s_1$, $L_3=s_3-s_2$, $L_4=s_4-s_3$):\up
\equa{
   k_1L_1+k_3L_2<\mu<k_2(L_1+L_2),\quad k_3L_3+k_1L_4<2\pi-\mu< k_4(L_3+L_4)\So
}
\Equa{RectMuNu}{%
    \max\Skobki{k_1L_1+k_3L_2,\, 2\pi-k_4(L_3+L_4)}
    <\mu<
    \min\Skobki{k_2(L_1+L_2),\, 2\pi-(k_3L_3+k_1L_4)}.
}
Restrictions for $\nu$ can be derived similarly.

\begin{figure}[t]
\centering%
\Pfig{\textwidth, bb=0 0 474 110}{Oval}{
Oval $V_0V_1V_2V_3V_0$, and modelling of possible locations of vertex $V_2$
}%
~~\\~~\\
\Pfig{\textwidth}{Closeness}{
Testing closeness of ovals~\eqref{NumExample}. 
Sets of endpoints (with one example of generating curve for each set),
obtained by varying parameter~$\mu$.
}%
~~\\~~\\
\Pfig{\textwidth}{NonClosed}{
In \eqref{NumExample} the value $k_1=0.25$ is replaced by $k_1=0.42$; sets do not intersect.
}%
\end{figure}

Let us subdivide oval $V_0V_1V_2V_3V_0$ into two curves, each with one vertex.\up
\begin{itemize}
\item
The first curve, $V_0V_1V_2$,
is traced from the coordinate origin along the direction $\nvec{-\frac{\pi}2}$.
Its turning $\mu$ is fixed, the set $A_1B_1C_1D_1A_1$ of possible endpoints~$V_2$
is shown in \RefFig[b]{Oval}.\up
\item
The second curve, $V_0V_3V_2$, is obtained by reversing the curve $V_2V_3V_0$.
Its curvature function is $\widetilde{k}(s)=-k(s_4{-}s)$, $0\le s\le s_4{-}s_2$.
The curve starts  from the coordinate origin along the direction $\nvec{\frac{\pi}2}$.
Its turning angle $\mu-2\pi$ is fixed such as to get smooth closeness,
if endpoints of the two curves come to coincidence.
The set $A_2B_2C_2D_2A_2$ of possible endpoints~$V_2$ of the second curve
is shown in \RefFig[c]{Oval}.
\end{itemize}

Non-empty intersection of two sets (\RefFig[d]{Oval}) means that,
with given~$\mu$, a common point~$V_2$ exists,
and constructing closed curve $V_0V_1V_2V_3V_0$ is possible.
Empty intersection would signify
that {\em with given~$\mu$ closeness is impossible}.
\RefFig{Closeness} with\up 
\Equa{NumExample}{
  k_1= 0.25,\;    L_1= 3.5,\quad 
  k_2= 0.80,\;    L_2= 3.2,\quad 
  k_3= 0.07,\;    L_3= 4.3,\quad 
  k_4= 0.85,\;    L_4= 4.0,
}
shows how the mutual position of two sets changes as~$\mu$ varies.
One of sets vanishes as~$\mu$ reaches limits~\eqref{RectMuNu}.
Two sets approach each other as~$\mu$ increases from minimum,
come to contact at \hbox{$\mu\approx 0.79\pi$}, intersect each other,
and diverge after \hbox{$\mu\approx 1.07\pi$}.

\enlargethispage{1\baselineskip}
A lot of tests discovered two kinds of behavior:\up
\begin{enumerate}
\item
The sets approach each other, intersect between two contact positions,
and separate thereafter.
Tangency of curves $AB$ and $CD$,
demonstrated in \RefFig{Closeness}, is not the only way of contact.\up
\item
Two sets pass around each other, without coming to contact (\RefFig{NonClosed}).
\end{enumerate}

Similar behavior was observed when parameter~$\nu$ was varied.
These observations lead to the following hypotheses about closeness of ovals.\up
\begin{enumerate}
\item
Closeness requires restrictions $\mu^\prime<\mu<\mu^{\prime\prime}$, $\nu^\prime<\nu<\nu^{\prime\prime}$,
essentially more narrow than natural restrictions~\eqref{RectMuNu}.
The values of $\mu^\prime,\mu^{\prime\prime},\nu^\prime,\nu^{\prime\prime}$
are solutions of some equations $F_j(x;L_i,k_i)=0$,
describing contact of boundary curves.\up
\item
The values $L_i,k_i$ exist, such that closeness of curve~\eqref{Candidate}
is impossible for whatever profile~$k(s)$.
We cannot specify whether this fact is reflected as
$\mu^{\prime\prime}\le \mu^\prime$ ($\nu^{\prime\prime}\le \nu^\prime$),
or as the absence of solutions of the above mentioned equations.\up
\item
These restrictions are {\em necessary for closeness} of the curve,
and {\em sufficient for existence} of an oval with given parameters $L_i,k_i$, $i=1,2,3,4$.
\end{enumerate}

With rather simple and clear graphic representation,
we come across bulky formulas,
which seems inevitable, as soon as high order cycloidal curves become involved.
E.\,g., bound $A_1B_1(u)$ is\up
\equa{
   \Exp{-\iu\frac{\pi}2}\cdot Z_2(q_1(u),L_1;\,q_2(u),L_2),\quad
      q_1(u)=\frac{\mu-2uL_2}{L_1+L_2}, \quad
      q_2(u)=\frac{\mu+2uL_1}{L_1+L_2}, \quad
    \begin{cases}k_1\le q_1(u) \le k_2,\\k_2\ge q_2(u)\ge k_3\end{cases}
}
($q_1L_1+q_2L_2=\mu$).
Bound $C_1D_1(v)$ is
\ $\Exp{-\iu\frac{\pi}2} Z_3(k_1,l_1(v);\,k_2,l_2(v);\,k_3,l_3(v))$,
with $Z_3$ defined as
\equa{%
      Z_3(q_1,l_1;\,q_2,l_2;\,q_3,l_3)=Z_1(q_1,l_1)+\Exp{\iu q_1 l_1}Z_1(q_2,l_2)+\Exp{\iu (q_1 l_1+q_2 l_2)}Z_1(q_3,l_3),
}
with $l_{1,2,3}(v)$ and the range for $v$ defined from
\equa{
  \begin{cases}
 l_1(v)+l_2(v)+l_3(v)=L_1+L_2,\\ k_1 l_1(v)+k_2l_2(v)+k_3l_3(v)=\mu,\\ l_3(v)-l_1(v)=2v;
  \end{cases}\qquad
  \begin{cases} 0\le l_1(v) \le L_1,\\ 0\le l_3(v)\le L_2.\end{cases}
}

We also note that these hypotheses got a concrete form
in the particular case\up
\equa{
  L_1=L_2=L_3=L_4=L,\qquad k_1=k_3\ge 0,\qquad k_2=k_4 > k_1.
}\up
Namely, denote\up
\equa{
    \varkappa_1=k_1L,\quad  \varkappa_2=k_2L,\qquad 
     \Phi(x;p,q)=q (x+2p)\sin\frac{x}2  -  x(q-p)\sin\frac{q(2p+x-2\pi)}{2(q-p)}.
}\up\up
The necessary condition of smooth closeness,
$\int\limits_{0}^{4L} k(s)\,\mathrm{d}s = 2\pi$,
requires for an oval
\equa{
    k_1 (L_1+L_4) + k_3 (L_2+L_3)<2\pi< k_2 (L_1+L_2) + k_4 (L_3+L_4)
    \So
    0\le \varkappa_1<\frac{\pi}2<\varkappa_2,
}
and
\eqref{RectMuNu} looks like
\Equa{NatLimits}{%
    \aligned
       \text{if~~}
       \varkappa_1+\varkappa_2\le\pi{:}\quad 
       && 2\pi-2\varkappa_2 < {}&\mu < 2\varkappa_2, 
      \\
       \text{if~~}
      \varkappa_1+\varkappa_2\ge\pi{:}\quad
       &&      2\varkappa_1< {}&\mu < 2\pi-2\varkappa_1  
    \endaligned
}
($\nu$ is in the same range).
We observe in this case the following.
\begin{enumerate}
\item
Intersection of two sets always occurs,
starting with tangency of bounds $A_2B_2$ and $C_1D_1$ at $\mu=\mu^\prime$,
ending with tangency of $A_1B_1$ and $C_2D_2$ at $\mu=\mu^{\prime\prime}$
(and likewise in modelling for~$\nu$).
Under the above parametrizations, $A_1B_1(u)$, $C_1D_1(v)$,
and similar ones for $A_2B_2(u)$, $C_2D_2(v)$,
tangencies occur at $u=0$ and $v=0$.\up
\item
Equations for $\mu',\mu'',\nu',\nu''$ are:
\equa{
     \Phi(\mu'';\varkappa_1,\varkappa_2)=0,\quad
     \Phi(\nu';\varkappa_2,\varkappa_1)=0,\quad
     \mu'=2\pi-\mu'',\quad  \nu''=2\pi-\nu'.
}
Numerical verification shows existence and uniqueness of solutions in ranges \eqref{NatLimits}.
\end{enumerate}


\section{Further work}

In conclusion, we announce some results, already studied by the author,
and concerning the further development of the subject.

Using approximation with spiral triarcs (instead of biarcs) reproduces,
together with boundary tangents, also boundary curvatures,
i.\,e. the two-point G$^2$ Hermite data of the original curve.
Approximation problem with length preservation
always has one or two solutions; the second one appears for arcs with inflection.

If the curve, split into $n$ spiral arcs,
is approximated by triarcs, we get~$3n$ circular arcs instead of~$2n$,
given by biarcs method.
But, for every inner node, the curvature of the left-sided arc is equal
to that of the right-sided arc, both being equal 
to the curvature~$k(s_i)$ of the original curve in the node.
These two arcs can be merged, and the approximating curve
will be composed of~$2n+1$ arcs
(or still of~$2n$, if the original curve was closed).

For approximation with  triarcs, two-level model~\eqref{Model1} is replaced by three-level one:
\equa{%
   \includegraphics{2to3.eps}
}   
The family of approximating curves remains three-parametric,
so, implementation of modelling procedure requires
the same depth of embedded loops as in the biarcs case.

Using triarcs approximation does not require convexity of the original arc.
Therefore, there is no need to include inflection points 
into initial splitting of a curve.
Taking into account applications to conformal maps,
mentioned in~\cite{Sabitov},
this feature resonates with the fact that,
e.\,g., under M\"{o}bius map of both the curve and its approximation,
verticez remain such~\cite[Cor.\,1.1]{InvInv}.
Inflection points are not thus invariant:
they can shift along the mapped curve, appear on, or disappear from it.


\end{document}